\documentclass[12pt]{amsart}
\usepackage{amssymb,amsmath,color}
\usepackage{dsfont}
\usepackage{enumitem,amsrefs,hyperref}
\usepackage{color}
\usepackage{tikz}
\usepackage{pgf}
\usepackage{bm}
\usepackage{tkz-graph}
\usetikzlibrary{calc, positioning}
\usepackage{fullpage}

\newtheorem{theorem}{Theorem}[section]

\newtheorem{lemma}[theorem]{Lemma}
\newtheorem{corollary}[theorem]{Corollary}
\newtheorem{proposition}[theorem]{Proposition}
\theoremstyle{definition}
\newtheorem{definition}[theorem]{Definition}
\newtheorem{example}[theorem]{Example}

\theoremstyle{remark}

\newcommand{\RR}{\mathbb{R}}

\newcommand{\A}{{\mathcal{A}}}

\newcommand{\Htwo}[3]{
\begin{tikzpicture}[baseline={([yshift=-.5ex]current bounding box.center)}]
\node[fill=black, circle, inner sep=1pt, minimum size=0.1cm, label=right:\footnotesize{#1}] (1) {};
\node[fill=black, circle, inner sep=1pt, minimum size=0.1cm, label=left:\footnotesize{#2}] (2) [below left = 0.3cm of 1] {};
\node[fill=black, circle, inner sep=1pt, minimum size=0.1cm, label=right:\footnotesize{#3}] (3) [below right =0.3cm of 1] {};
\draw (3)--(1)--(2);
\end{tikzpicture}}

\newcommand{\uHtwo}[0]{
\begin{tikzpicture}[baseline={([yshift=-.5ex]current bounding box.center)}]
\node[fill=black, circle, inner sep=1pt, minimum size=0.1cm] (1) {};
\node[fill=black, circle, inner sep=1pt, minimum size=0.1cm] (2) [below left = 0.3cm of 1] {};
\node[fill=black, circle, inner sep=1pt, minimum size=0.1cm] (3) [below right =0.3cm of 1] {};
\draw (3)--(1)--(2);
\end{tikzpicture}}

\newcommand{\threestarplus}[4]{
\begin{tikzpicture}[baseline={([yshift=-.5ex]current bounding box.center)}]
\node[fill=black, circle, inner sep=1pt, minimum size=0.1cm, label=right:\footnotesize{#1}] (1) {};
\node[fill=black, circle, inner sep=1pt, minimum size=0.1cm, label=left:\footnotesize{#2}] (2) [above = 0.3cm of 1] {};
\node[fill=black, circle, inner sep=1pt, minimum size=0.1cm, label=left:\footnotesize{#3}] (3) [below left = 0.3cm of 1] {};
\node[fill=black, circle, inner sep=1pt, minimum size=0.1cm,label=right:\footnotesize{#4}] (4) [below right = 0.3cm of 1] {};
\draw (2)--(1)--(3)--(1)--(4);
\draw (3)--(4);
\end{tikzpicture}}

\newcommand{\uthreestar}[0]{
\begin{tikzpicture}[baseline={([yshift=-.5ex]current bounding box.center)}]
\node[fill=black, circle, inner sep=1pt, minimum size=0.1cm] (1) {};
\node[fill=black, circle, inner sep=1pt, minimum size=0.1cm] (2) [above = 0.3cm of 1] {};
\node[fill=black, circle, inner sep=1pt, minimum size=0.1cm] (3) [below left = 0.3cm of 1] {};
\node[fill=black, circle, inner sep=1pt, minimum size=0.1cm] (4) [below right = 0.3cm of 1] {};
\draw (2)--(1)--(3)--(1)--(4);
\end{tikzpicture}}

\newcommand{\vedge}[2]{
\begin{tikzpicture}[baseline={([yshift=-.5ex]current bounding box.center)}]
\node[fill=black, circle, inner sep=1pt, minimum size=0.1cm,label=right:\footnotesize{#1}] (1) {};
\node[fill=black, circle, inner sep=1pt, minimum size=0.1cm,label=right:\footnotesize{#2}] (2) [below = 0.3cm of 1] {};
\draw (2)--(1);
\end{tikzpicture}}

\newcommand{\uvedge}[0]{
\begin{tikzpicture}[baseline={([yshift=-.5ex]current bounding box.center)}]
\node[fill=black, circle, inner sep=1pt, minimum size=0.1cm] (1) {};
\node[fill=black, circle, inner sep=1pt, minimum size=0.1cm] (2) [below = 0.3cm of 1] {};
\node[draw=none, fill=none, circle, inner sep=1pt, minimum size=0.1cm] (4) [left = 0.05cm of 1] {};
\draw (2)--(1);
\end{tikzpicture}}

\newcommand{\upthree}[0]{
\begin{tikzpicture}[baseline={([yshift=-.5ex]current bounding box.center)}]
\node[fill=black, circle, inner sep=1pt, minimum size=0.1cm] (1) {};
\node[fill=black, circle, inner sep=1pt, minimum size=0.1cm] (2) [above = 0.3cm of 1] {};
\node[fill=black, circle, inner sep=1pt, minimum size=0.1cm] (3) [right = 0.3cm of 2] {};
\node[fill=black, circle, inner sep=1pt, minimum size=0.1cm] (4) [below = 0.3cm of 3] {};
\draw (1)--(2)--(3)--(4);
\end{tikzpicture}}

\newcommand{\sidconj}[0]{
\begin{tikzpicture}[baseline={([yshift=-.5ex]current bounding box.center)}]
\tikzstyle{every node}=[fill=black, circle, inner sep=1pt, minimum size=0.1cm]
\draw (1*360/10: 1.5cm) node (1) [label=above right:$e$] {};
\draw (2*360/10: 1.5cm) node (2) [label=above:$d$] {};
\draw (3*360/10: 1.5cm) node (3) [label=above:$c$] {};
\draw (4*360/10: 1.5cm) node (4) [label=above left:$b$] {};
\draw (5*360/10: 1.5cm) node (5) [label=left:$a$] {};
\draw (6*360/10: 1.5cm) node (6) [label=below left:$j$] {};
\draw (7*360/10: 1.5cm) node (7) [label=below:$i$] {};
\draw (8*360/10: 1.5cm) node (8) [label=below:$h$] {};
\draw (9*360/10: 1.5cm) node (9) [label=below right:$g$] {};
\draw (10*360/10: 1.5cm) node (10) [label=right:$f$] {};
\draw (1)--(2)--(3)--(4)--(5)--(6)--(7)--(8)--(9)--(10)--(1);
\draw (1)--(6);
\draw (2)--(7);
\draw (3)--(8);
\draw (4)--(9);
\draw (5)--(10);
\end{tikzpicture}}

\newcommand{\bottomlesshouse}[5]{
\begin{tikzpicture}[baseline={([yshift=-.5ex]current bounding box.center)}]
\node[fill=black, circle, inner sep=1pt, minimum size=0.1cm, label=right:\footnotesize{#1}] (1) {};
\node[fill=black, circle, inner sep=1pt, minimum size=0.1cm, label=left:\footnotesize{#2}] (2) [below left = 0.3cm of 1] {};
\node[fill=black, circle, inner sep=1pt, minimum size=0.1cm, label=right:\footnotesize{#3}] (3) [below right = 0.3cm of 1] {};
\node[fill=black, circle, inner sep=1pt, minimum size=0.1cm, label=left:\footnotesize{#4}] (4) [below = 0.3cm of 2] {};
\node[fill=black, circle, inner sep=1pt, minimum size=0.1cm, label=right:\footnotesize{#5}] (5) [below = 0.3cm of 3] {};
\draw (1)--(3)--(2)--(1);
\draw (2)--(4);
\draw (3)--(5);
\end{tikzpicture}}

\date{\today}

\title[Simple Graph Density Inequalities without Sos Proofs]{Simple Graph Density Inequalities\\ with no Sum of Squares Proofs}
\thanks{Grigoriy Blekherman was partially supported by NSF grant DMS-1352073. This material is partially based upon work supported by the National Science Foundation under Grant No. 1440140, while Annie Raymond and Rekha Thomas were in residence at the Mathematical Sciences Research Institute in Berkeley, California, during the fall of 2017. 
Mohit Singh was partially supposed by NSF grant CCF-1717947 and Rekha Thomas was partially supported by 
NSF grant DMS-1719538.}

\author{Grigoriy Blekherman}
\address{School of Mathematics, Georgia Institute of Technology,
686 Cherry Street
Atlanta, GA 30332}\email{greg@math.gatech.edu}

\author{Annie Raymond}
\address{Department of Mathematics and Statistics,
Lederle Graduate Research Tower, 1623D,
University of Massachusetts Amherst
710 N. Pleasant Street
Amherst, MA 01003} \email{raymond@math.umass.edu}

\author{Mohit Singh}
\address{H. Milton Stewart School of
Industrial and Systems Engineering, Georgia Institute of Technology,
755 Ferst Drive, NW, Atlanta, GA 30332}
\email{mohitsinghr@gmail.com}

\author{Rekha R. Thomas}
\address{Department of Mathematics, University of Washington, Box
  354350, Seattle, WA 98195, USA} \email{rrthomas@uw.edu}

\begin{document}

\begin{abstract}
Establishing inequalities among graph densities is a central
pursuit in extremal combinatorics. A standard tool to certify the nonnegativity of a graph density expression
is to write it as a sum of squares. In this paper, we identify a simple condition under which a graph density expression cannot be a sum of squares. Using this result, we prove that the
Blakley-Roy inequality does not have a sum of squares certificate when the path length is odd. 
We also show that the same Blakley-Roy inequalities cannot be certified by sums of squares using a multiplier of the form one plus a sum of squares. These results answer two questions raised by Lov\'asz. Our main tool is used again to show that the smallest open case of Sidorenko's conjectured inequality cannot be certified by a sum of squares.
Finally, we show that our setup is equivalent to existing frameworks by Razborov and Lov\'asz-Szegedy, and thus our results hold in these settings too.\end{abstract}

\maketitle

\section{Introduction}

A graph $G$ has vertex set $V(G)$ and edge set $E(G)$. All graphs are assumed to be simple, without loops or multiple edges. The \emph{homomorphism density} of a graph $H$ in a graph $G$, denoted by $t(H;G)$, is the probability that a random map from $V(H)$ to $V(G)$ is a graph homomorphism, i.e., it maps every edge of $H$ to an edge of $G$.  An inequality between homomorphism densities refers to an inequality between $t(H_i;G)$, for some finite graphs $H_i$, that is valid for all graphs $G$.

Many results and problems in extremal graph theory can be restated as inequalities between homomorphism densities~\cite{LovaszBook,Razborov07}. The Cauchy-Schwarz inequality has been one of the powerful tools used to verify density inequalities for graphs and hypergraphs~\cites{MR3002572,LovaszBook, MR2680226,MR3007147}.  This proof method is equivalent to the general sum of squares (sos) proof method that has been widely used in optimization~\cite{blekherman2012semidefinite}. Moreover, sos proofs naturally yield to a computerized search via semidefinite programming. It was shown in ~\cite{MR2921000} that every true inequality between homomorphism densities is a limit of Cauchy-Schwarz inequalities.

On the other hand, Hatami and Norine~\cite{HN11} show significant computational limitations on verifying inequalities between homomorphism densities.
Firstly, they show that the problem of verifying the validity of an inequality between homomorphism densities is undecidable. Moreover, they also show that there are valid linear inequalities between graph  homomorphism densities that do not have a finite sos proof.

Despite the above negative results, the limitations of the sos proof method in proving a particular graph density inequality of interest has been unclear. The examples arising from \cite{HN11} do not shed much light on natural graph density inequalities in extremal graph theory.

In this paper, we give a simple criterion that rules out sos proofs for the validity of a given graph density inequality. As a corollary of our method, we obtain that certain classical graph density inequalities cannot be proven via the sos method. Moreover, we also show that the smallest unresolved instance of the celebrated Sidorenko's conjecture cannot be resolved via the sos method.

 To describe our results, we begin with a few definitions about the gluing algebra of graphs. We refer the reader to Lov\'{a}sz~\cite{LovaszBook} for a broader exposition. A graph is {\em partially labeled} if a subset of its vertices are labeled with elements of $\mathbb{N} := \{1,2,3,\ldots\}$ such that no vertex receives more than one label. If no
vertices of $H$ are labeled then $H$ is {\em unlabeled}.

Let $\A$ denote the vector space of all formal finite $\RR$-linear combinations of partially labeled graphs without isolated vertices, including
the empty graph with no vertices which we denote as $1$. We call an element $a = \sum \alpha_i H_i$ of $\A$ a {\em graph combination}, each $\alpha_i H_i$ a {\em term} of $a$, and each $H_i$ a {\em constituent graph}  of $a$.  The {\em degree} of a term $\alpha_i H_i$, $\alpha_i \neq 0$, is the number of edges in $H_i$. We say that $a$ is {\em homogeneous} of degree $d$ if all its terms have degree $d$.

Let $\A_\emptyset$ denote the subalgebra of $\A$ spanned by unlabeled graphs. We view elements $a \in \A_\emptyset$ as functions that can be evaluated on unlabeled graphs $G$ via homomorphism densities. An element $a = \sum \alpha_i H_i$ of $\A_\emptyset $ is called nonnegative if $ \sum \alpha_i t(H_i;G)\geq 0$ for all graphs $G$.

The vector space $\A$ has a product defined as follows. For two labeled graphs   $H_1$ and $H_2$, form the new labeled graph $H_1H_2$  by gluing together the vertices in the two graphs with the same label, and keeping only one copy of any edge that may have doubled in the process. Equipped with this product, $\A$ becomes an $\RR$-algebra with the empty graph as its multiplicative identity.

The algebra $\A$ admits a simple linear map into $\A_\emptyset$ that removes the labels in a graph combination to create a graph combination of unlabeled graphs. We call this map {\em unlabeling} and denote it by $[[ \cdot ]]$. 

A {\em sum of squares (sos)} in $\A_\emptyset$ is a finite sum of unlabeled squares of graph combinations $a_i \in \A$, namely,
$\sum [[ a_i^2 ]]$. It can be easily seen that a sos is a nonnegative graph combination.

\begin{definition} \label{def:trivial square}
An unlabeled graph $F$ is called a {\em trivial square} if whenever $F = [[H^2]]$ for some labeled graph $H$, then $H$ is a fully labeled copy of $F$.
\end{definition}

In Lemma~\ref{lem:involution}, we give a characterization of trivial squares in terms of automorphisms of the underlying graph. Our main result is the following theorem that gives a sufficient condition for when a graph combination is not a sos.

\begin{theorem} \label{thm:main1}
Let $f = \sum_{s=1}^t \lambda_s F_s$ be a graph combination of unlabeled graphs $F_s$ and $d_{\text{min}}$ be the minimum degree of any $F_s$. Suppose $f$ satisfies the following conditions:
\begin{enumerate}
  \item there exists an $s$ such that the degree of $F_s$ is equal to $d_{\text{min}}$ and $\lambda_s<0$, and
    \item for every $s$ such that degree of $F_s$ equals $d_{\text{min}}$ and $\lambda_s>0$, $F_s$ is a trivial square.
\end{enumerate} Then $f$ is a not a sos.
\end{theorem}

As a first application of our theorem, we consider the Blakley-Roy inequality~\cite{blakley1965holder}. Let $P_k$ denote the path of length $k$ and $e$ denote the graph with a single edge. Then the Blakley-Roy inequality asserts that for every $k$, the combination $P_k-e^k$ is nonnegative. Indeed various proofs of this inequality have been obtained, for instance~\cite{blakley1965holder,li2011logarithimic}. We show the following result.

\begin{corollary} \label{cor:BR}
For any odd integer $k\geq 3$, $P_k$ is a trivial square. Therefore, for every odd $k\geq 3$ and for all $\lambda\in \mathbb{R}$, $\lambda P_k-e^k$ is not a sos.
\end{corollary}

The above result answers Question 17(b) in Lov\'{a}sz~\cite{LovaszOpenProblems} which asked whether the  Blakley-Roy inequality has a sos proof. 
In ~\cite{LovaszOpenProblems} Lov\'{a}sz also considered a more general certificate of nonnegativity: it is easy to see that $f$ is a nonnegative graph combination if there exists a sos graph combination $g$ such that $f(1+g)$ is sos.
Theorem~\ref{thm:main1} shows that for homogeneous graph combinations, such nonnegativity certificates are no more powerful than usual sos. In particular, the Blakley-Roy inequality for odd paths cannot be certified in this way.

\begin{corollary} \label{cor:multipliers}
For any $\lambda\in \mathbb{R}$ and for any $k \geq 3$ and odd, there is no sos $g$ such that $(\lambda P_k - e^k)(1+g)$ is a sos.
\end{corollary}

This resolves question 21 of ~\cite{LovaszOpenProblems} which asked for an explicit example of a valid homomorphism density inequality without such multiplicative certificates. The existence of such inequalities already followed from the undecidability result of \cite{HN11}.

As a final corollary, we consider Sidorenko's conjecture~\cite{Sid93} that states that for every bipartite graph $H$, the graph combination $H-e^{|E(H)|}$ is nonnegative. A special case of this conjecture, which is known to be true, is the Blakely-Roy inequality. 
While it has been verified for various graph families, the smallest $H$ for which the conjecture remains open is $H=K_{5,5}\setminus C_{10}$ where $K_{5,5}$ is the complete bipartite graph where both parts contain five vertices, and $C_{10}$ is a Hamiltonian cycle with $10$ vertices \cite{conlon2018sidorenko}. We show that Theorem~\ref{thm:main1}
 implies that the above inequality cannot be resolved using the sos proof 
 method.

\begin{corollary} \label{cor:sidorenko}
If $H=K_{5,5}\setminus C_{10}$, then $H$ is a trivial square, and $H-e^{15}$ is not a sos. Moreover, there is no sos $g$ such that $(H-e^{15})(1+g)$ is a sos. 
\end{corollary}

Our main technical tool is Lemma \ref{lem:square in lowest degree}, which shows that for any sos  $f = \sum \lambda_s F_s = \sum [[a_i^2]]$ with $a_i = \sum \alpha_{ij} H_{ij}$, there exists a term $\lambda_s F_s$ of minimal degree in $f$ such that $F_s$ only
arises as a square $[[H_{ij}^2]]$, and consequently, $\lambda_s > 0$. From this, we derive Theorem \ref{thm:homogeneous sos}, which shows that in decomposing homogeneous graph combinations as sums of squares, we are severely restricted in the types of graphs that can be used in the underlying squares. In forthcoming work we will show how these restrictions can be used to classify all homogeneous sums of squares of degrees 3 and 4 \cite{BRSThom}.

This paper is organized as follows. In Section \ref{sec:gluing} we prove our main results on sums of squares in the gluing algebra $\mathcal{A}$. In Section \ref{sec:translation} we discuss the relation between our gluing algebra and the Cauchy-Schwarz calculus of Razborov as well as the very closely related gluing algebras of Lov\'{a}sz-Szegedy. These connections prove that the three results presented in Corollaries~\ref{cor:BR}, \ref{cor:multipliers} and \ref{cor:sidorenko} also hold in any of these settings.

{\bf Acknowledgments.} We thank Prasad Tetali for bringing to our attention the smallest open case of Sidorenko's conjecture. We also want to thank Alexander Razborov for useful discussions about this paper.


\section{The Gluing Algebra and its Sums of Squares}\label{sec:gluing}

Recall the algebra $\mathcal{A}$ from the introduction spanned by partially labeled graphs as a $\mathbb{R}$-vector space. We call
$\mathcal{A}$ a gluing algebra since multiplication in it works by gluing graphs along vertices with the
same labels. For example, $\threestarplus{3}{}{1}{2} \cdot \Htwo{1}{2}{}= \bottomlesshouse{2}{3}{1}{}{}$.
For a fixed finite set of labels $L \subset \mathbb{N}$, let $\A_L$ denote the subalgebra of $\A$ spanned by all
graphs whose label sets are contained in $L$. Then $\A_\emptyset$ is the subalgebra of $\A$ spanned by unlabeled graphs.

\begin{lemma} \label{lem:degrees are same}
Let $H_1$ and $H_2$ be two partially labeled graphs such that $$\deg(H_1 H_2) \leq \min \{ \deg(H_1^2), \deg(H_2^2) \}.$$ Then $\deg(H_1 H_2) = \deg (H_1^2) = \deg(H_2^2)$. Further, $\deg(H_1) = \deg(H_2)$ and
$H_1$ and $H_2$ have the same set of fully labeled edges.
\end{lemma}

\begin{proof} Suppose $H_i$ has degree $d_i$ and $l_i$ fully labeled edges. Then $\deg(H_i^2) = 2d_i - l_i$.
Let $c$ be the number of fully-labeled edges that are common to both $H_1$ and $H_2$. Then $c \leq \min\{l_1, l_2\}$ and $\deg(H_1 H_2) = d_1 + d_2 - c$.

We are given that
$d_1 + d_2 - c \leq \min \{ 2d_1 -l_1, 2d_2 - l_2 \}$, which implies that
$$ 0 \leq l_1 - c \leq d_1 - d_2 \,\,\,\textup{ and } \,\,\, 0 \leq l_2 - c \leq d_2 - d_1.$$
The extremes of the two inequalities give that $d_1 = d_2$, while adding the two inequalities
gives that $l_1 + l_2 - 2c = 0$. Since $c \leq \min \{l_1, l_2\}$,  it follows that $l_1 = l_2 = c$.
\end{proof}

The unlabeling map $[[ \cdot ]] \,:\, \A \rightarrow \A_\emptyset$ removes the labels in a graph combination.
Note that for any partially labeled graph $H$, $\deg(H) = \deg([[H]])$. A sum of squares (sos) in $\A$ is a finite sum of unlabeled squares of graph combinations $a_i \in \A$, namely,
$\sum [[ a_i^2 ]]$. By definition, a sos in $\A$ lies in $\A_\emptyset$.

\begin{example}\label{ex:mult}
The Blakley-Roy inequality for a path of length two, $\uHtwo-\uvedge \uvedge \geq 0$, has a sum of squares proof as follows (\cite{LovaszBook}, pages 28-29).

$$ [[(\vedge{\textcolor{white}{1}}{1}-\uvedge)^2]] =  [[\Htwo{1}{}{}-2\ \vedge{}{} \vedge{\textcolor{white}{1}}{1}+\uvedge\uvedge \ ]]= \uHtwo-\uvedge\ \uvedge$$
\end{example}

We will now investigate the structure of homogeneous
graph combinations that are sos. Their properties and limitations are the key ingredients in the proof of our main results.
We begin with the following lemma whose proof will be postponed to the end of this section.

\begin{lemma} \label{lem:nonzero binomial square and pointed cone}
\begin{enumerate}
\item If $F$ and $H$ are two partially labeled graphs such that $[[(F-H)^2]] = 0$, then $F=H$.
\item Suppose $\sum_{i=1}^k \alpha_i[[(F_i-H_i)^2]]$=0 with $\alpha_i\geq0$ and $F_i \neq H_i$ for each $i$, then  $\alpha_i=0$ for all $i$.
\end{enumerate}
\end{lemma}

\begin{lemma} \label{lem:square in lowest degree}
Let $f = \sum \lambda_s F_s = \sum [[a_i^2]]$ be a sos in $\A$ with $a_i = \sum \alpha_{ij} H_{ij}$. Let $d$ be the minimum degree of any cross product $H_{ij}  H_{ik}$ within any $a_i$. Then
there exists a term $\lambda_s F_s$ in $f$ of degree $d$ such that $F_s$ only
arises via squares $[[H_{ij}^2]]$, and consequently, $\lambda_s > 0$.
\end{lemma}

\begin{proof}
Let $\A_\emptyset^d$ denote the vector space spanned by all unlabeled graphs of degree $d$. Let $C$ be the cone in $\A_\emptyset^d$ generated by all unlabeled squares of the form $[[(F-H)^2]]$ where $F$ and $H$ are distinct partially labeled graphs such that $\deg(F^2) = \deg(H^2) = \deg(FH) = d$. Since $d$ is fixed,
there are only finitely many possibilities for the generators $[[(F-H)^2]]$ of $C$ and hence $C$ is polyhedral, and therefore, closed. Furthermore, $C$ is pointed by Lemma~\ref{lem:nonzero binomial square and pointed cone} (2). Since $C$ is closed and pointed, its dual cone $C^\ast$  is full-dimensional in $(\A_\emptyset^d)^\ast$. Therefore, we may pick a sufficiently generic linear functional $L \,:\, \A_\emptyset^d \rightarrow \mathbb{R}$ from the interior of $C^\ast$ that will not only have the property that $L(a) > 0$ for all nonzero $a \in C$, but also takes distinct values on the finitely many unlabeled graphs of degree $d$ in $\A_\emptyset^d$.

Consider all distinct graphs $[[H_{ij}   H_{ik}]]$ of degree $d$ that can be formed by multiplying two constituent graphs in any $a_i$ and then unlabeling, including the unlabeled squares $[[H_{ij}^2]]$.
Let $F$ be the unique largest graph in this list in the total order induced by $L$, and suppose $F = [[H_{ij}   H_{ik}]]$ for some $i$ and $j \neq k$. By Lemma~\ref{lem:degrees are same} we have that
$d = \deg(F) = \deg(H_{ij}^2) = \deg(H_{ik}^2)$. Since $H_{ij} \neq H_{ik}$, by Lemma~\ref{lem:nonzero binomial square and pointed cone} (1), $[[(H_{ij} - H_{ik})^2]]$ is a nonzero generator
of the cone $C$, and hence, $L([[(H_{ij} - H_{ik})^2]]) > 0$. This implies that
$L([[H_{ij}^2]])  + L([[H_{ik}^2]]) > 2 L([[H_{ij}   H_{ik}]])$. Therefore, at least one of $L([[H_{ij}^2]])$ or
$L([[H_{ik}^2]])$ is strictly greater than $L(F)$ which contradicts the choice of $F$. Thus $F$ only arises via squares of the form $[[H_{ij}^2]]$ for some $H_{ij}$. Therefore, it must be a constituent graph of $f$ and setting $F_s=F$ proves the lemma.
\end{proof}

\begin{corollary} \label{cor:lowest degree part is sos}
Let $f = \sum [[a_i^2]]$ be a sos in $\A$ with $a_i = \sum \alpha_{ij} H_{ij}$ and let $d$ be the lowest degree of a term in $f$. Then the degree $d$ component of $f$ is again a sos, $\sum [[c_i^2]]$, where all cross products of terms in each $c_i$ have degree $d$.
\end{corollary}

\begin{proof}
By Lemma~\ref{lem:square in lowest degree}, we know that $d$ is the lowest degree of any cross product $[[H_{ij}   H_{ik}]]$ of graphs $H_{ij}, H_{ik}$ in an $a_i$. Let $f_d$ be the degree $d$ component of $f$. For each $i$, let $b_i$ denote the graph combination obtained from $a_i$ by deleting all terms $\alpha_{ij} H_{ij}$ for which $\deg[[H_{ij}^2]] > d$. By Lemma~\ref{lem:degrees are same}, a deleted term $\alpha_{ij} H_{ij}$ from $a_i$ could not have  cross multiplied with another term $\alpha_{ik} H_{ik}$ in $a_i$ to produce a term of degree $d$. Therefore, the
degree $d$ component of $\sum [[b_i^2]]$ is precisely $f_d$.

Suppose $G,H,K$ are three partially labeled graphs in some $b_i$ with
$\deg(GH) = \deg(HK) = d$. Then by Lemma~\ref{lem:degrees are same},
$\deg(G^2) = \deg(H^2) = \deg(K^2) = d$ and $G,H,K$ all have the same set of fully labeled edges. Let $c$ be the number of fully labeled edges in $G,H,K$.
Then $d = \deg(G^2) = 2 \deg(G) - c$ which implies that $c = 2 \deg(G) - d$. Similarly,
$c = 2 \deg(K) - d$ and hence $\deg(G) = \deg(K)$.
Therefore, $\deg(GK) = \deg(G) + \deg(K) - c = 2\deg(G) - 2 \deg(G) + d = d$.
To conclude, if $\deg(GH) = \deg(HK) = d$, then also $\deg(GK) = d$. This means that we may define an equivalence relation by saying $G \sim H$ if $\deg(GH) = d$.

Now group the terms in each $b_i$ so that all constituent graphs in a group are equivalent in the above sense.  Suppose the graph combinations
corresponding to each group are $b_{i1}, b_{i2}, \ldots, b_{it_i}$. By construction, all cross products of terms in any $b_{ij}$ have degree $d$. Consider the new sos expression
$g := \sum_i \sum_{j=1}^{t_i} [[b_{ij}^2]]$. By construction, $\deg(g) = d$. For each $i$, all terms in $\sum_{j=1}^{t_i} [[b_{ij}^2]]$ occur among the terms of $[[b_i^2]]$. By our regrouping of terms in a $b_i$, a
term in the expansion of $[[b_i^2]]$ is absent from $\sum_{j=1}^{t_i} [[b_{ij}^2]]$ if and only if its degree is larger than $d$. Therefore, $g = f_d$, and we have obtained an sos expression $\sum [[c_i^2]] := \sum_i \sum_{j=1}^{t_i} [[b_{ij}^2]]$ for $f_d$ of the desired form.
\end{proof}

We illustrate the previous corollary with the following example.

\begin{example}\label{ex:lemma}
$$f=2\uthreestar + \uvedge\uvedge\uvedge - 2\upthree + 2 \uthreestar\uvedge - 2\upthree\uvedge= [[(\Htwo{2}{1}{}-\Htwo{1}{2}{}-\vedge{1}{3}\uvedge)^2]]$$

Here, $d=3$ and $f_3=2\uthreestar + \uvedge\uvedge\uvedge - 2\upthree$. Following the procedure from Corollary~\ref{cor:lowest degree part is sos}, we obtain two equivalence classes $\{\Htwo{2}{1}{}, \Htwo{1}{2}{}\}$ and $\{\vedge{1}{3}\uvedge\}$. Therefore, we see that $$f_3=[[(\Htwo{2}{1}{}-\Htwo{1}{2}{})^2]]+[[(\vedge{1}{3}\uvedge)^2]].$$
\end{example}

The above results prove an important structural property of homogeneous graph combinations that are sos which we record in the following theorem. This property will play a crucial role in this paper.

\begin{theorem}\label{thm:homogeneous sos}
Every homogeneous graph combination of degree $d$ that is a sos has a sos expression of the form  $\sum [[a_i^2]]$ where all cross products of terms in any $a_i$ have degree $d$.
\end{theorem}

Recall the definition of a trivial square of Definition~\ref{def:trivial square}. We now prove a characterization of trivial squares using automorphisms of the graph.

\begin{lemma} \label{lem:involution} An unlabeled connected graph $G$ is a non-trivial square, i.e., $G = [[F^2]]$ for some partially labeled graph $F$ that is not a fully labeled copy of $G$, if and only if there is an automorphism $\varphi$ of $G$ such that
\begin{enumerate}
\item $\varphi$ is an involution,
\item $\varphi$ fixes a non-empty proper subset of the vertices of $G$, and
\item the vertices $v$ that are not fixed by $\varphi$ can be partitioned into two groups,
each group consisting of one vertex from the pair $\{v, \varphi(v)\}$, such that there are no edges
between vertices in the two groups.
\end{enumerate}
\end{lemma}

\begin{proof} Suppose $G = [[F^2]]$ is a non-trivial square. Consider the map $\varphi \,:\, V(G) \rightarrow V(G)$ that fixes all vertices of $G$ that were labeled in $F^2$, and sends an unlabeled vertex $v$ of $F^2$ to $v'$ where $v$ and $v'$ are copies of the same unlabeled vertex in $F$. Then $\varphi$ is an automorphism of $G$ that is also an involution.
Since $G$ is connected, $F$ is not unlabeled, and it is not fully labeled by assumption.
Therefore, $\varphi$ fixes a non-empty proper subset of the vertices of $G$.
The first group of vertices in (3) is made up of the unlabeled vertices $v$ in $F$ and the second
group is made up of the duplicates $v'$ of unlabeled vertices in $F$ that exist in $F^2$.

Conversely, suppose $G$ has an automorphism $\varphi$ with properties (1)-(3).
Then consider the graph
$F$ obtained by identifying vertices $v$ and $\varphi(v)$ and deleting a second copy of any edge that gets doubled in this process. Observe that Property (3) ensures that no edge is more than doubled, and no loops are
created, and therefore $F$ is a simple graph. Label all vertices of $F$ that identified with themselves with distinct labels to create a partially labeled graph $\tilde{F}$.
It follows by construction that $G = [[\tilde{F}^2]]$. Since $\varphi$ fixes a non-empty proper set of vertices of $G$, only a proper set of vertices of $\tilde{F}$ are labeled, hence
$G$ is a non-trivial square.
\end{proof}

\begin{example}[Paths] \label{ex:paths}
Let $P_k$ be an unlabeled path with $k$ edges and $k+1$ vertices $v_1,\dots,v_{k+1}$.
Using Lemma~\ref{lem:involution} one can argue that a path $P_k$ of odd length $k$ is a trivial square.
Every graph automorphism $\varphi$ of $P_k$ has to send $v_1$ to either itself or to $v_{k+1}$. Each choice
completely determines $\varphi$ since adjacent vertices have to be sent to adjacent vertices. If $v_1$ is sent to $v_1$ then
$\varphi$ is the trivial involution that fixes all vertices in $P_k$. If $v_1$ is sent to $v_{k+1}$ then $v_2$ is sent to $v_{k}$, etc until $v_{k+1}$ is sent to $v_1$.
This involution doesn't fix any vertices in $P_k$. Either way, we see from Lemma~\ref{lem:involution} that
$P_k$ is a trivial square.

On the other hand, if $k$ is even, then $P=[[F^2]]$ where $F$ is a path of length $\frac{k}{2}$ with the first vertex labeled $1$.
\end{example}

The notion of trivial squares together with Theorem~\ref{thm:homogeneous sos} will provide us with a tool to recognize homogeneous graph combinations that are not sos.

We can now prove Theorem~\ref{thm:main1} from the introduction.\\

\begin{proof}[Proof of Theorem~\ref{thm:main1}]
Let $f_{d_{min}}$ be the (homogeneous) lowest degree component of $f$. By
Corollary~\ref{cor:lowest degree part is sos}, if $f$ is a sos, then $f_{d_{min}}$ is a sos. By Theorem~\ref{thm:homogeneous sos},
 $f_{d_{min}} = \sum [[a_i^2]]$ where all unlabeled cross products of constituent graphs in each $a_i$ have degree $d_{min}$. We also know from Lemma~\ref{lem:square in lowest degree} that one of the $F_s$ in $f_{d_{min}}$ only arises as $[[H_{ij}^2]]$ for some $H_{ij}$ in some $a_i$ and then $\lambda_s > 0$. By assumption, whenever $\lambda_s > 0$, $F_s$ is a trivial square, which means that $F_s = [[H_{ij}]]$ for some $i$ and $j$ and $H_{ij}$ is fully labeled.

Pick a trivial square $F_s$ in $f_{d_{min}}$ and suppose $F_s = [[H_{ij}]]$.
For the same $i$ and $j$, consider the cross product $[[H_{ij}   H_{ik}]]$ for some $j \neq k$. Since $H_{ij} \neq H_{ik}$ as partially labeled graphs, their product has degree larger than $d$ which is a contradiction. So it must be that $H_{ij}$ is the only constituent graph of $a_i$ and $a_i = \alpha_{ij}H_{ij}$. Therefore, we may remove all occurrences of fully labeled graphs that square and unlabel to $F_s$ from the sos decomposition $\sum [[a_i^2]]$ to get an sos expression for
$f' = f_{d_{min}} - \lambda_s F_s$. Repeating this procedure, we may remove all trivial squares from $f_{d_{min}}$ to get a graph combination $\bar{f}$ with only negative coefficients that is still an sos. This is a contradiction since
we showed in Lemma~\ref{lem:square in lowest degree} that a sos always has a term with a positive coefficient.
\end{proof}

\bigskip

We will now apply Theorem~\ref{thm:main1} to prove our main results.
The first application is to show that the Blakley-Roy inequality $P_k - e^k \geq 0$, for $k \geq 3$ and odd, cannot be certified by sums of squares in the gluing algebra.
It follows that this result provides a
negative answer to Problem 17 (b)  in \cite{LovaszOpenProblems} which asked whether the Blakley-Roy inequality has a sos certificate (see Section~\ref{sec:translation} for details).

\begin{proof}[Proof of Corollary~\ref{cor:BR}]
If $\lambda \leq 0$ then $\lambda P_k - e^k$ is not sos by Lemma~\ref{lem:square in lowest degree} which says that every homogenous sos has a term with a positive coefficient.
If $\lambda > 0$, the result follows from Theorem~\ref{thm:main1} and Example~\ref{ex:paths} which showed that $P_k$ is a trivial square.

\end{proof}

Problem 21 in \cite{LovaszOpenProblems} asks the general question as to whether it is always possible to
certify the nonnegativity of a graph combination $f$ by multiplying it with $(1+g)$ where
$g$ is sos and having the product be a sos?  It was shown in \cite{HN11} that the answer is no.
We provide the first explicit example of this by showing that for $f = \lambda P_k - e^k$ there are no sos $g \in \A_\emptyset$ such that
$f(1+g)$ is a sos. Using results of Section~\ref{sec:translation}, it will follow that this
answers Lov\'asz's question negatively.

\begin{proof}[Proof of Corollary~\ref{cor:multipliers}]
Let $f=\lambda P_k-e^k$ for any $\lambda\in \mathbb{R}$ where $k$ is odd which we just showed is not a sos. Suppose there was a sos $g \in \A_\emptyset$ such that $f(1+g)$ is sos.
Then the lowest degree part of $f(1+g)$ is precisely $f$ which is not a sos.
This contradicts Corollary~\ref{cor:lowest degree part is sos}
which says that the lowest degree part of a sos is again sos.

\end{proof}

Sidorenko's conjecture is that $H-e^{|E(H)|} \geq 0$ when $H$ is a bipartite graph.
Note that $P_k - e^k \geq 0$ is an instance of this and has several proofs as mentioned in the introduction.
The smallest open case of Sidorenko's conjecture is to establish the inequality for $H = K_{5,5} \setminus C_{10}$ where $K_{5,5}$ is the complete bipartite graph with two color classes of size five and $C_{10}$ is a Hamiltonian cycle through the $10$ vertices of $K_{5,5}$. Our tools show that it is not possible to use sos to establish the nonnegativity of $H - e^{|E(H)|}$ when
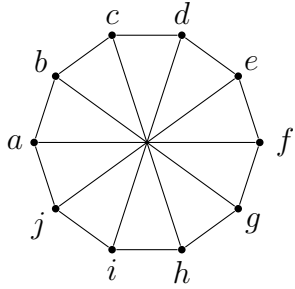
$H = K_{5,5} \setminus C_{10}$.

\begin{center}
\sidconj \\
\medskip
$H = K_{5,5} \setminus C_{10}$ labeled as in the proof of Corollary~\ref{cor:sidorenko}
\end{center}

\bigskip
\begin{proof}[Proof of Corollary~\ref{cor:sidorenko}]
We use Lemma~\ref{lem:involution} to show that $H = K_{5,5} \setminus C_{10}$ is a trivial square. The conclusion follows from Theorem~\ref{thm:main1} with the same argument as in Corollary \ref{cor:multipliers}.

We first argue that the automorphism group of $H$ is the dihedral group $D_{10}$ which is the automorphism group of the $10$-cycle $C_{10}$. Observe that there is a unique (up to swapping colors) two-coloring of $H$. Any involution of $H$ either permutes the vertices within each color class, or swaps the two color classes. The complement $\bar{H}$ of $H$ is a $10$-cycle with two complete graphs $K_5$ on the even and odd vertices respectively. Any automorphism $\varphi$ of $H$ is also an automorphism of $\bar{H}$. The above argument shows that $\varphi$ sends edges of the union of the two $K_5$'s to themselves. Therefore $\varphi$ sends the edges of the $10$-cycle to itself, and thus the automorphism group of $H$ is a subgroup of $D_{10}$. However, it is easy to see that automorphisms of the $10$-cycle are also automorphisms of $H$, and the automorphism group of $H$ is $D_{10}$.

There are $11$ involutions in $D_{10}$, ten of which are reflections and one is rotation by $180$ degrees. We want to argue that each of these involutions violate at least one of the properties (1)-(3) of
Lemma~\ref{lem:involution}. Five of the reflections and the rotation do not fix any vertices which violates property (2). The remaining five reflections about the diagonals are involutions that fix two vertices of $H$. We will argue that these reflections violate property (3). It suffices to argue this for one of them. Consider the reflection of $H$ about its horizontal diagonal. This involution fixes vertices $a$ and $f$, but sends $b \mapsto j$, $c \mapsto i$,
$d \mapsto h$ and $e \mapsto g$. Now we check whether the vertices that are not fixed by the involution can be partitioned as in (3). We see that vertices $b,c,d,e$ have to be in the same group since
there are edges connecting one to the next. But then $g,h,i,j$ also belong to this group because of the diagonals.
Thus it is not possible to divide the vertices that are not fixed by the involution into two groups as in (3), and $H$ is a trivial square.
\end{proof}

\bigskip

We note that the proof of Lemma~\ref{lem:square in lowest degree} says something special about the
sos decomposition of homogeneous graph combinations of the form $F_1 - F_2$. Recall the cone $C$ from the proof of the lemma that was generated by
unlabeled squares of the form $[[(F-H)^2]]$ where $\deg(F^2) = \deg(H^2) = \deg(FH) = d$. Since $C$ is polyhedral, it has a finite inequality description which allows one to test for membership in $C$.

\begin{proposition} A graph combination $f = F_1 - F_2$ where $F_1, F_2$ are two unlabeled graphs of the same degree is a sos
if and only if it has a sos decomposition of the form $\sum \lambda_{ij} [[(H_i - H_j)^2]]$ with $\lambda_{ij} \geq 0$.
\end{proposition}

\begin{proof}
Suppose $f$ is a sos but $f \not \in C$. Since $C$ is pointed,
there is a linear functional $L$ such that
$L(f) < 0$ and $L(c) > 0$ for all nonzero $c \in C$. Since $L(F_1) < L(F_2)$, the proof of Lemma~\ref{lem:square in lowest degree} says that $F_2$ is a square and its coefficient in $f$ is
positive which is a contradiction.  Therefore, $f$ lies in $C$ which means that
$f = \sum \lambda_{ij} [[(H_i-H_j)^2]]$ for some $\lambda_{ij} \geq 0$.
\end{proof}

The rest of this section is devoted to the proof of Lemma~\ref{lem:nonzero binomial square and pointed cone}. This needs the notion of nonnegativity of a graph combination which we saw briefly in the introduction.
For this, we will need to view elements $a \in \A$ as functions that can be evaluated on unlabeled graphs $G$, including those with isolated vertices. We closely follow the exposition in \cite{HN11}.

Recall that a graph homomorphism between two unlabeled graphs $H$ and $G$ is an adjacency preserving map $h \,:\, V(H) \rightarrow V(G)$ such that
$h(i)h(j) \in E(G)$ if $ij \in E(H)$. The homomorphism density of $H$ in $G$, denoted as $t(H;G)$ is the
probability that a random map from $V(H) \rightarrow V(G)$ is a homomorphism. Define
$t(1;G) := 1$ for all $G$. Now suppose $H$ is a partially labeled graph and $L_H$ is its set of labels.
Given a map $\varphi: L_H \rightarrow V(G)$, define the homomorphism density $t(H;G,\varphi)$ as the
probability that a random map from $V(H)$ to $V(G)$ is a homomorphism conditioned on the
event that the labeled vertices in $H$ are mapped to $V(G)$ according to $\varphi$. Then, by the rules of conditional probability, $t([[H]];G)$ is the (positively) weighted average of the conditional probabilities $t(H;G,\varphi)$ over all maps $\varphi$.
For a combination of partially labeled graphs $a = \sum_{i=1}^t \alpha_iH_i$, let $L_a = \cup_{i=1}^t L_{H_i}$ be the union of all label sets of all constituent graphs of $a$. Then for a
fixed map $\varphi \,:\, L_a \rightarrow V(G)$,  define $t(a;G,\varphi) := \sum \alpha_i t(H_i;G,\varphi|_{L_{H_i}})$.

Suppose we now fix a label set $L$ and a map $\varphi \,:\, L \rightarrow V(G)$. Then
if $H_1$ and $H_2$ are two partially labeled graphs whose label sets $L_{H_1}$ and $L_{H_2}$ are contained in $L$,
$t(H_1H_2;G, \varphi) = t(H_1;G,\varphi|_{L_{H_1}}) t(H_2;G,\varphi|_{L_{H_2}})$. Recall that $\A_L$
was the subalgebra of $\A$ consisting of all partially labeled graphs whose label sets are contained in $L$. Then we have that
$t(-;G,\varphi)$ is a homomorphism from $\A_L$ to $\mathbb{R}$.

We say that $a \in \A$ is {\em nonnegative} if $t(a;G,\varphi) \geq 0$ for all
unlabeled graphs $G$ and maps $\varphi \,:\, L_a \rightarrow V(G)$.
Note that any partially labeled graph $H$ is nonnegative
since $t(H;G,\varphi)$ is a probability. By the same reason, $H^2$ and $[[H^2]]$ are also nonnegative, but graph combinations $a \in \A$ are not necessarily nonnegative since they have
arbitrary coefficients. However, if $a \in \A$, then $a^2$ is nonnegative since
$t$ is a homomorphism. In particular, any sos $\sum [[a_i^2]] \in \A$ is nonnegative.

To prove Lemma~\ref{lem:nonzero binomial square and pointed cone}, we will need the notion of {\em weighted graph homomorphisms} as in \cite[\S 5.2]{LovaszBook}. A {\em node-weighted graph} $G$ is one with node weights $\omega_u(G)$ on the nodes $u \in V(G)$. To a map $h \,:\, V(H) \rightarrow V(G)$, define
$\omega_h(H,G) := \prod_{u \in V(H)} \omega_{h(u)}(G)$. The number of weighted homomorphisms  (resp. weighted injective homomorphisms) from
$H$ to $G$ is then $\sum \omega_h(H,G)$ where the sum varies over all homomorphisms $h \,: \,nonzero V(H) \rightarrow V(G)$ (resp. all injective homomorphisms $h \,:\, V(H) \rightarrow V(G)$).

\bigskip
\begin{proof}[Proof of Lemma~\ref{lem:nonzero binomial square and pointed cone}]
\begin{enumerate}

\item We need to show that if $[[(F-H)^2]]=0$ then $F=H$. Suppose $0 = [[(F-H)^2]]=[[F^2]]+[[H^2]]-2[[FH]]$.
Then $[[F^2]] = [[H^2]] = [[FH]]$ since density functions of unlabeled graphs are linearly independent \cite[Corollary~5.45]{LovaszBook}.
This implies that $F$ and $H$ must have the same number of vertices, and the same set of labels $L$. Indeed, you see this by counting the number of vertices in each of the three graphs $[[F^2]], [[H^2]]$ and $[[FH]]$, differentiated by how many are labeled in $F$ and $H$, and how many labels are shared between $F$ and $H$.
If for a graph $G$ and a map $\varphi: L\rightarrow V(G)$, $t(F;G, \varphi) \neq t(H;G, \varphi)$,
then $t((F-H)^2;G,\varphi)  >0$ which implies that $t([[(F-H)^2]],G) > 0$, contradicting that $[[(F-H)^2]]=0$. Our proof strategy will be to show that if $F \neq H$ then there is a graph $G$ and a map $\varphi \,:\, L \rightarrow V(G)$ such that $t(F;G, \varphi) \neq t(H;G, \varphi)$.

In order to construct the graph $G$, we follow the proof of \cite[Proposition 5.44]{LovaszBook}. Weigh the unlabeled vertices in both $F$ and $H$ with distinct variables $x_i$ and call these weighted graphs $\tilde{F}$ and $\tilde{H}$.
Then consider the $2\times 2$ matrix $M$ with rows indexed by $F$ and $H$, with columns indexed by
$\tilde{F}$ and $\tilde{H}$ and entries equal to the number of weighted homomorphisms from $F,H$ to
$\tilde{F}, \tilde{H}$, where the (commonly) labeled vertices of $F$ (resp.\ $H$) and $\tilde{F}$ (resp.\ $\tilde{H}$) map to each other.
The matrix $M$ is filled with polynomials and hence, $\det(M)$ is a polynomial. We observe that $\textup{rank}(M) = 2$, i.e., $\det(M)$ is not identically zero. Indeed, since the variables $x_i$ are all distinct,
the multilinear component of $\det(M)$ is $\det(M')$, where $M'$ is the matrix with entries equal to the number of
weighted injective homomorphisms from $F$ (resp.\ $H$) to $\tilde{F}$ (resp.\ $\tilde{H}$),
which is nonzero, since $M'$ is upper/lower triangular, and its diagonal entries are nonzero polynomials.
Hence, $\det(M)$ is not the zero polynomial, which shows that for an algebraically independent
substitution of the $x_i$'s, $\det(M)$ will not vanish.

This means that there is some choice of positive integer values for the $x_i$'s that will keep the matrix
$M$ non-singular. Substitute each $x_i$ with such a positive integer to get a new matrix $M$ and replace a vertex $v$ of weight $m$ in $\tilde{F}$ or $\tilde{H}$  with $m$ copies of itself (each with the same neighborhood that $v$ had) to obtain graphs $G_1$ and $G_2$.
The matrix with entries equal to the number of homomorphisms from $H$ and $F$, to $G_1$ and $G_2$, such that
the commonly labeled vertices are mapped to each other is still $M$. Here $\varphi$ is again the map that sends $L$ to $V(G_1)$ and $V(G_2)$. Converting these entries to homomorphism
densities involves dividing each column in the matrix by a constant which keeps the resulting matrix
again non-singular.  Therefore, its rows are not scalar multiples of each other, and either $t(F;G_1,\varphi) \neq t(H;G_1,\varphi)$, or $t(F;G_2,\varphi) \neq t(H;G_2,\varphi)$.

\item For $\lambda_i \geq 0$,
$\sum \lambda_i [[(F_i - H_i)^2]] = 0$ if and only if for each $i$, either $\lambda_i = 0$ or
$[[(F_i-H_i)^2]] = 0$. We have that $[[(F-H)^2]] =0$ if and only if $F=H$.  Since $F_i \neq H_i$ for each $i$, it must be that $\lambda_i = 0$ for all $i$.
\end{enumerate}
\end{proof}
\bigskip


\section{Translation of our gluing algebra to other settings} \label{sec:translation}

The goal of this section is to show that the existence of a sum of squares certificate is equivalent in Razborov's flag algebra and Lov\'asz-Szegedy's gluing algebra and Hatami-Norine's gluing algebra, and the gluing algebra we presented in Section~\ref{sec:gluing}. Therefore Corollaries \ref{cor:BR}, \ref{cor:multipliers} and \ref{cor:sidorenko} hold in these settings as well.

\subsection*{Relation to Lov\'{a}sz-Szegedy.} A family of gluing algebras very similar to ours was introduced in the work of Lov\'{a}sz and Szegedy \cite{MR2921000}. The only difference is that they allowed combinations of graphs with isolated vertices, and the graphs in each algebra were required to have the same labels. However, the resulting unlabeled sos expressions are the same up to removing isolated vertices, which does not affect homomorphism densities.

Let $f$ be a graph combination in $\A_\emptyset$. Suppose that we have a sos expression $f=\sum [[a_i^2]]$ where each $a_i$ is in $\A_L$ where $L=[k]=\{1,\dots,k\}$ and where $a_i=\sum_{j=1}^t \alpha_{ij} H_{ij}$. For each constituent partially labeled graph $H_{ij}$ with label set $L_{H_{ij}}$, define $\bar{H}_{ij}$ to be the graph $H_{ij}$ with $k-|L_{H_{ij}}|$ labeled isolated vertices attached, where new vertices are labeled with the labels from $[k]\backslash L_{H_{ij}}$. We thus obtain a sos $\sum \bar{a}^2_i$ in the Lov\'{a}sz-Szegedy algebra of $k$-labeled quantum graphs. Observe that after unlabeling the expressions, $\sum [[a_i^2]]$ and $\sum [[\bar{a}^2_i]]$ are equivalent up to adding or removing isolated vertices. 

Similarly, start with a sos expression in the algebra of $k$-labeled quantum graphs $\sum \bar{a}^2_i$. For a constituent partially $k$-labeled graph $\bar{G}$ define $G$ to be the partially labeled graph obtained from $\bar{G}$ by removing the isolated vertices. This gives a sum of squares expression in $\mathcal{A}$, which agrees with the Lov\'{a}sz-Szegedy sum of squares up to removing isolated vertices.

\subsection*{Relation to Hatami-Norine.}

A variant of Lov\'{a}sz-Szegedy gluing algebra was defined by Hatami and Norine in \cite{HN11}. Our gluing algebra $\mathcal{A}$ is isomorphic to the quotient algebra in their paper. The partially labeled graph $H$ with no isolated vertices is just an explicit coset representative of the quotient by the ideal $\mathcal{K}$  generated by all differences of the form empty graph minus $1$-vertex graph with a label (or unlabeled). We refer to \cite{HN11} for more details.

\subsection*{Relation to Razborov's Flag Algebras.}

A different algebra was used in the work of Razborov in \cite{Razborov07}. There, partially labeled graphs are called \emph{flags}. The main difference is that flag algebras are concerned with \textit{induced subgraph density}, while homomorphism density is known to be asymptotically equal to \textit{non-induced subgraph density}. A well-known M\"{o}bius transformation relates induced and non-induced subgraph densities via a change of basis. Multiplication in the flag algebras looks syntactically different from the gluing algebra, however after passing through the M\"{o}bius transformation and its inverse, the two multiplications are the same. Therefore Cauchy-Schwarz proofs in the flag algebras are equivalent to sos proofs in the gluing algebra. We refer to \cite{RSST} and \cite{RST} for more details.

\bibliographystyle{alpha}
\bibliography{references}{}

\end{document}